\documentclass[12pt,reqno]{article}

\usepackage[usenames]{color}
\usepackage{amssymb}
\usepackage{graphicx}
\usepackage{amscd}

\usepackage[colorlinks=true,
linkcolor=webgreen,
filecolor=webbrown,
citecolor=webgreen
]{hyperref}

\definecolor{webgreen}{rgb}{0,.5,0}
\definecolor{webbrown}{rgb}{.6,0,0}

\usepackage{color}
\usepackage{float}

\usepackage{graphics,amsmath,amssymb}
\usepackage{amsthm}
\usepackage{amsfonts}
\usepackage{latexsym}

\begin{document}

\theoremstyle{plain}
\newtheorem{theorem}{Theorem}
\newtheorem{corollary}[theorem]{Corollary}
\newtheorem{lemma}[theorem]{Lemma}
\newtheorem{proposition}[theorem]{Proposition}
\newtheorem{obs}[theorem]{Observation}

\theoremstyle{definition}
\newtheorem{definition}[theorem]{Definition}
\newtheorem{example}[theorem]{Example}
\newtheorem{conjecture}[theorem]{Conjecture}
\newtheorem{question}[theorem]{Question}

\theoremstyle{remark}
\newtheorem{remark}[theorem]{Remark}

\begin{center}
\vskip 1cm
{\LARGE\bf Multilattice graphs and perfect domination}

\vskip 1cm
\large
I. Dejter\\
University of Puerto Rico\\
Rio Piedras, PR 00936-8377\\
\href{mailto:italo.dejter@gmail.com}{\tt italo.dejter@gmail.com} \\
L. Fuentes\\
Universidad Nacional Abierta y a Distancia\\
Cartagena, Colombia\\
\href{mailto:luis.fuentes@unad.edu.co}{\tt luis.fuentes@unad.edu.co}\\
and\\
C. Araujo\\
Universidad del Atlantico\\
Barranquilla, Colombia\\
\href{mailto:carlosaraujo@mail.uniatlantico.edu.co}{\tt carlosaraujo@unad.edu.co}\\
\end{center}

\begin{abstract}\noindent
Perfect codes in the $n$-dimensio\-nal grid $\Lambda_n$ of the lattice $\mathbb{Z}^n$ ($0<n\in\mathbb{Z}$) and its quotient toroidal grids were obtained via the truncated distance in $\mathbb{Z}^n$ given between $u=(u_1,\cdots,u_n)$ and $v=(v_1, \ldots,v_n)$ as the graph distance $h(u,v)$ in $\Lambda_n$, if $|u_i-v_i|\le 1$, for all $i\in\{1, \ldots,n\}$, and as $n+1$, otherwise. 
Such codes are extended to multilattice graphs $\Gamma_n$ obtained by glueing ternary $n$-cubes along their codimension 1 ternary subcubes in such a  way that each binary $n$-subcube is contained in a unique maximal lattice of $\Gamma_n$.  The existence of an infinite number of isolated perfect truncated-metric codes of radius 2  in $\Gamma_n$ for $n=2$ is ascertained, leading to conjecture such existence for $n>2$ with radius $n$.
\end{abstract}

\section[\bf 1.]{\bf Perfect truncated-metric codes}\label{ss3}

Our motivations are algorithm-flowchart modeling that applies to alternate reality games, as well as computer architecture associated to signal transmission in a finite field (cf. \cite{E,Etzion,BH}).
Related topics of coding theory and lattice domination are based on the Lee metric arising from the Minkowsky $\ell_p$ norm \cite{Campello} with $p=1$.
These topics, tied to the still unsolved Golomb-Welch Conjecture~\cite{GW,HG,50}, were investigated via perfect Lee codes \cite{E}, diameter perfect Lee codes \cite{BH}, tilings with generalized Lee spheres \cite{BE,Etzion}, perfect dominating sets (PDS's) \cite {W}, perfect-distance dominating sets (PDDS's) \cite{ADH} and efficient dominating sets \cite{DS}.

Let $0<n\in\mathbb{Z}$. The aforementioned codes and dominating sets are realized in the tiling of the lattice $\mathbb{Z}^n$ whose tiles are translates of a common generalized Lee sphere. This tiling happens in
 the $n${\it -dimensional grid} $\Lambda_n$, namely the {\it lattice graph} whose vertex set is $\mathbb{Z}^n$ with exactly one edge between each two vertices at euclidean distance 1. A  natural follow-up here is to consider tilings of $\mathbb{Z}^n$ with two different generalized Lee spheres, which for $n>3$ is only possible by modifying the notion of Lee distance $d$ to that of a {\it truncated distance} $\rho$, defined below (not a standard distance, but see Remark~\ref{ell}), having perfect code applications (Subsection~\ref{s2}, based on the rainbow-coloring results of \cite{rainbow}).
 Furthermore, {\it multilattice graphs} $\Gamma_n$, that are compounds of ternary $n$-cubes glued along codimension 1 ternary subcubes (Remark~\ref{ter} and Section~\ref{s4}) generalizing the lattice graphs $\Lambda_n$ (even though the $\Gamma_n$ are not lattice graphs themselves), allow to extend the mentioned perfect code applications (Section~\ref{s5}). The graphs $\Gamma_n$ can be taken as `` alternate reality" graphs since they offer from each vertex of a ternary $n$-cube glued in the compound $\Gamma_n$ two ternary options as arrows indicating coordinate directions.

The {\it Hamming distance} $h(u,v)$ between vectors $u=(u_1,\cdots,u_n)$ and $v=(v_1, \ldots,v_n)$ in $\mathbb{Z}^n\subset\mathbb{R}^n$ is the number of positions $i\in[n]=\{1, \ldots,n\}$ for which $u_i\ne v_i$.
Now, let $\rho:\mathbb{Z}^n\times\mathbb{Z}^n\rightarrow\mathbb{Z}$ be given by:
\[\rho(u,v)=\begin{cases}
h(u,v),&\mbox{ if }|u_i-v_i|\le 1\mbox{ for all }i\in[n]; \\
n+1,& \mbox{otherwise.}\\
\end{cases}
\]

Given $S\subseteq\mathbb{Z}^n$, $[S]$ is the induced subgraph of $S$ in $\Lambda_n$. In the definition of $\rho$, the Hamming distance $h$ is equivalent to the graph distance in $\Lambda_n$, which is extensible to the graphs $\Gamma_n$, where
 algorithm flowcharts can be adapted following ternary directions in ternary $n$-cubes, with vertices representing decision conditional diamonds that lead to ``Yes" or ``True" and ``No" or ``False" arrows, i.e. two oriented edges representing two sides of a triangle in a ternary $(n-1)$-cube, each such arrow representing one or more process rectangles. 
An elementary example of this is given in Section~\ref{appendix}.

Assign to each component $H$ of $[S]$
an integer $t_H>0$ to be understood as the {\it radius} of a {\it truncated sphere with center} $H$. For every component $H'$ of $[S]$ in the translation class $\langle H\rangle$ of $H$ in $\mathbb{Z}^n$, we assume $t_{H'}=t_H$. This yields a correspondence $\kappa$ from the set  of translation classes $\langle H\rangle$ of components $H$ of $[S]$ into $\mathbb{Z}$ such that $\kappa(\langle H\rangle)=t_H$, for every $H\in\langle H\rangle$.

Let $\rho(u,S)=\min\{\rho(u,s)|s\in S\}$. Let $(H)^{\kappa(\langle H\rangle)}=\{u\in\mathbb{Z}^n|\rho(u,H)\le\kappa(\langle H\rangle)\}$.
A $\kappa$-{\it perfect truncated-metric code}, or {\it $\kappa$-PTMC}, is a set $S\subseteq\mathbb{Z}^n$ such that for every $u\in\mathbb{Z}^n$ there exists a unique $s\in S$ with $\rho(u,s)=\rho(u,S)$ and such that
the collection $$\{(H)^{\kappa(\langle H\rangle)}|H\mbox{ is a connected component of }S\}$$ forms a partition of $\mathbb{Z}^n$.
For each component $H$ of $S$, $(H)^{\kappa(\langle H\rangle)}$ is the {\it truncated sphere centered at} $H$  (or {\it around} $H$) {\it with radius} $\kappa(\langle H\rangle)$.
Here, $H$ is the {\it truncated center of} $(H)^{\kappa(\langle H\rangle)}$ and its vertices are the {\it central vertices of} $H$.

Given $0<t\in\mathbb{Z}$, a $\kappa$-PTMC is a {\it $t$-PTMC} if all $\kappa(\langle H\rangle)$'s are equal to $t$.
If $|H|=1$, then $|(H)^t|=\sum_{i=0}^t2^i{n\choose i}$, for $0\le t\le n$, (which is less than the cardinal of the corresponding $\ell_1$ sphere of radius $t$, for $t>1$ \cite{Campello}). We note that 1-PTMC's coincide with PDS's \cite{W}.

\begin{remark}\label{ter} A {\it binary} (resp. {\it ternary}) $n$-{\it cube graph} $Q_n^i$, succinctly called $n$-{\it cube} (resp. $n$-{\it tercube}), is a cartesian product $K_i\square\cdots\square K_i$ of $n$ complete graphs $K_i$ with $V(K_i)=\{0, \ldots,i-1\}=F_i$, the field of order $i=2$ (resp. $3$). In Sections~\ref{s4}-~\ref{s5}, we extend the study of PTMC's from the graphs $\Lambda_n$ to the graphs $\Gamma_n$; 
these cannot have {\it non-isolated PDS}'s, that is {\it non-isolated 1-PTMC}'s; (however, the ternary perfect single-error-correcting codes of length $\frac{3^t-1}{2}$ \cite{xxxx,BE} are {\it isolated PDS}'s, that is {\it efficient dominating sets} \cite{DS} in the ternary $\frac{3^t-1}{2}$-cubes; these also are edge-disjoint unions of triangles, for $t>0$). For $G=\Gamma_n$ we will say in Sections~\ref{s4}-\ref{s5} that: an edge-disjoint union $G$ of triangles has a {\it non-isolated PDS} $S$ if every vertex of $G$ not in $S$ is adjacent to just either one vertex of $G$ (as in {\it PDS}'s) or the two end-vertices of an edge of $S$. 
As a consequence, a non-isolated PDS is not necessarily a PDS. Section~\ref{s5} shows that $\Gamma_2$ has an infinite number of isolated 2-PTMC and we conjecture that $\Gamma_n$ has isolated $n$-PTMC's, $\forall n>2$.
\end{remark}

\begin{figure}[htp]
\includegraphics[scale=0.73]{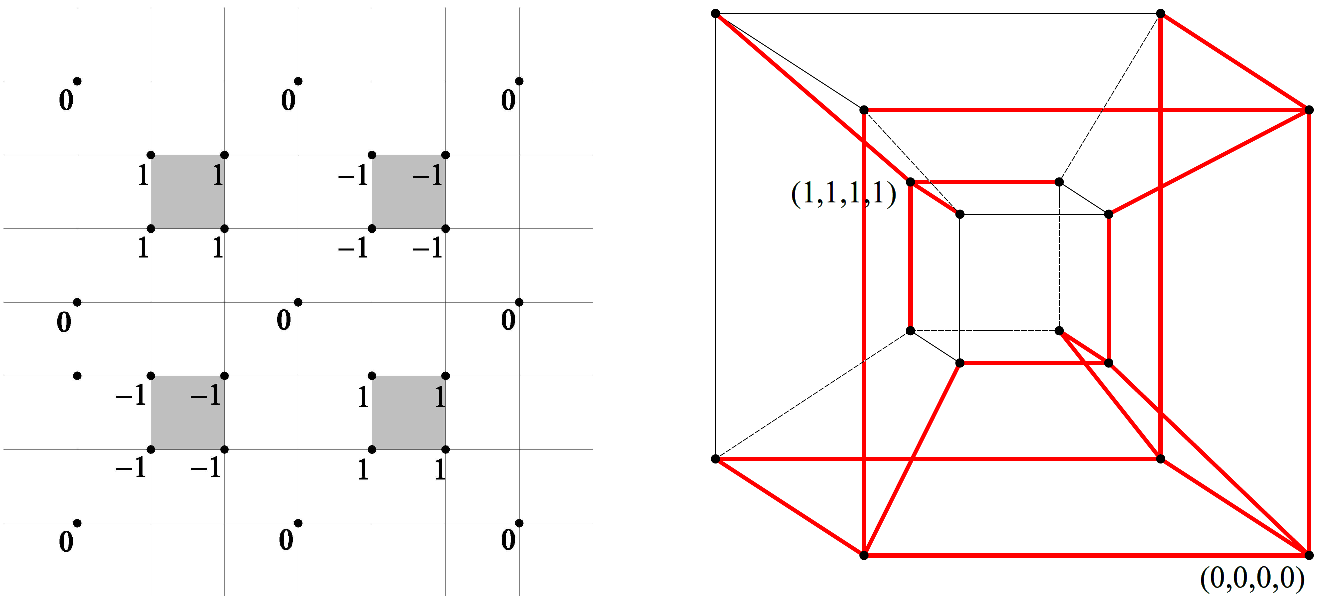}
\caption{Example accompanying Remark~\ref{apedido}}
\end{figure}

\begin{remark}\label{apedido}
 Some graph-coloring results of \cite{rainbow} are restated below in Theorem~\ref{t2}-\ref{t4} in terms of PTMC's asserting the existence of PTMC's in $\Lambda_n$ that are {\it lattice}  \cite{BH}, that is {\it lattice-like} \cite{ADH}, having their induced components $H$ with vertex sets {\bf(a)} whose convex hulls are $n$-parallelotopes (resp. both either 0-cubes and $(n-1)$-cubes); {\bf(b)} contained in, and dominating, truncated spheres centered at those $H$ with radii $n$ (resp. radius $n-2$ and 1) that are the tiles of an associated lattice tiling of $\mathbb{Z}^n$; for $n=3$, 
 this is schematized on the left of Figure 1 that represents the assignment of the last coordinate value mod 3 on the projection of the PTMC onto $\mathbb{Z}^{n-1}$; the case $n=4$ 
 (the first PTMC which is neither PDS nor PDDS) can be visualized in the same fashion, where the dominating edges, red colored in the copy of $Q^2_4\subset\mathbb{Z}^4$ on the right of Figure 1 indicate truncated spheres of radius 2 and 1 centered respectively at $(0,0,0,0)$ and $(1,1,1,1)$.
\end{remark}

\begin{remark}\label{re} We observe that: 
{\bf(I)}
 a $t$-PTMC (resp. a truncated $t$-sphere) of $\Lambda_n$ is defined like a PDDS \cite{ADH} (resp. generalized Lee sphere \cite{Etzion}) with $\rho$ instead of the Lee distance $d$; {\bf(II)} 1-PTMC's of $\Lambda_n$ (resp. truncated 1-spheres) coincide with PDS's \cite{W} (resp. generalized Lee 1-spheres).\
 {\bf(III)} an example \cite{ADH} cited in Remark~\ref{ref1}, below, justifies that our definition of PTMC above employs translation classes instead of isomorphism classes, for in such example the truncated spheres have common radius (namely 1) for both translation classes, but that does not exclude the eventuality of an example with differing radii.\end{remark}

\begin{remark}\label{motiv}
Related to Theorems~\ref{t3}-\ref{t4} on building ``lattice'' $\kappa$-PTMC's whose induced components are $r$-cubes of different dimensions $r$ ($0\le r\le n$), we have:
{\bf(A)} the perfect covering codes with spheres of two different radii in Chapter 19 \cite{Cohenetal} and
{\bf(B)}
a negative answer \cite{OW} to a conjecture \cite{W} claiming that the induced components of every 1-PTMC $S$ in an $n$-cube $Q^2_n$ are all $r$-cubes $Q^2_r$ (not necessarily in the same translation class \cite{DP}), with $r$ fixed.
In fact, it was found in \cite{OW} that a 1-PTMC in $Q^2_{13}$ whose induced components are $r$-cubes $Q^2_r$ of two different values $r=r_1$ and $r=r_2$ exists, specifically for $r_1=4$ and $r_2=0$. This seems to be the only known counterexample to the cited conjecture \cite{W}.
\end{remark}

\subsection[\bf 1.1.]{\bf Notation, $\ell_p$ metrics and periodicity}\label{pre}

If no confusion arises, every $(a_1, \ldots,a_n)\in\mathbb{Z}^n$ is expressed as $a_1\cdots a_n$. Let $O=00\cdots 0$, $e_1=10\cdots 0$, $e_2=010\cdots 0$, $\ldots$, $e_{n-1}=0\cdots 010$ and $e_n=00\cdots 01$.

Let $z\in\mathbb{Z}^n$, let $H=(V,E)$ be a subgraph of $\Lambda_n$ and let $H+z$ be the subgraph $H'=(V',E')$ whose vertex set is
$V'=V+z=\{w\in\mathbb{Z}^n;\exists\; v\in V$ such that  $w=v+z\}$ so that $uv\in E$ $\Leftrightarrow$ $(u+z)(v+z)\in E'$. For example, let $H$ be induced in $\Lambda_n$ by the vertices with entries in $\{0,1\}$. Then, every translation $H+z$ of $H$ in $\Lambda_n$ (in particular $H$ itself) is isomorphic to $Q^2_n$.

Let $i\in[n]$.
Each edge (resp. segment) $uv$ of $\Lambda_n$ (resp. $\mathbb{R}^n\supset\mathbb{Z}^n$) with $u\ne v$ is {\it parallel} to $Oe_i\in E(\Lambda_n)
\Leftrightarrow u-v\in\{\pm e_i\}$ (resp. $\Leftrightarrow\exists a\in\mathbb{R}\setminus\{0\}$ with $u-v=ae_i$). A {\it parallelotope} $\mathcal P$ in $\mathbb{R}^n$ is a cartesian product of segments parallel to some of the $Oe_i$'s ($i\in[n]$). We restrict to $\mathcal P$'s having their vertices in $\mathbb{Z}^n$. An {\it edge} of $\mathcal P$ is a segment of unit length parallel to some $Oe_i$ ($i\in[n]$) separating a pair of vertices of $\Lambda_n$ in $\mathcal P$. Recall that a {\it trivial} subgraph of $\Lambda_n$ is composed by just one vertex. The proof of the following is similar to that of Theorem 1 \cite{ADH}.

\begin{theorem}\label{rain} Let $0<t\in\mathbb{Z}$.
If $S$ is a $t${\rm -PTMC} in $\Lambda_n$, then the convex hull in $\mathbb{R}^n$ of the vertex set $V(H)$ of each nontrivial component $H$ of $[S]$ is a parallelotope in $\mathbb{R}^n$ whose edges are parallel to some or all of the segments $Oe_i$ in $E(\Lambda_n)$, where $i\in[n]$.
\end{theorem}

Let $H$ be a nontrivial induced subgraph of $\Lambda_n$ with the convex hull of $V(H)$ in $\mathbb{R}^n$ as a parallelotope whose edges are parallel to some or all of the segments $Oe_i$ in $E(\Lambda_n)$, ($i\in[n]$).
If exactly $r$ elements of $[n]$ are such values of $i$, then we say that $H$ is an $r$-{\it box}. Note that $H$ is a cartesian product $\Pi_{i=1}^n P^i$, where $P^i$ is a finite path, $\forall i\in[n]$, with exactly $r$ paths $P^i$ having positive length. 

Let $W_{n,H,t}$ be the truncated sphere of radius $t$ around an $H$ as above, where $0\le t\in\mathbb{Z}$. Then, $H$ is the truncated center of $W_{n,H,t}$.
A $t$-PTMC $S$ of $\Lambda_n$ determines a partition of $\mathbb{Z}^n$ into spheres $W_{n,H,t}$ with $H$ running over the components of $[S]$.
Such an $S$ with the components of $[S]$ obtained by translations from a fixed finite graph $H$ is said to be a {\it $t$-PTMC$[H]$}. (This imitates the definition of a $t$-PDDS$[H]$ in \cite{ADH}, mentioned in Remark~\ref{re}(I) above).
Let $S$ be a $t$-PTMC$[H]$ and let $H'$ be a component of $[S]$ obtained from $H$ by means of a translation. Then $S$ is said to be a {\it lattice $t$-PTMC$[H]$} if and only if there is a lattice $L\subseteq\mathbb{Z}^n$ such that: $H''$ is a component of $[S]$ $\Leftrightarrow$
$H''=H'+z$, for some $z\in L$.

\begin{remark}\label{ell} It is relevant to note the connections of $\rho$ and the $\ell_p$ metrics \cite{Campello}: A truncated sphere of truncated radius 1 centered at some vertex $v$ is a Lee  ($p=1$) sphere of radius 1, whereas a truncated sphere of truncated radius $n$ centered at $v$ is a sphere of radius 1 in the maximum ($p=\infty$) distance, namely $W_{n,\{v\},1}$
(which has an $n$-dimensional cube for convex hull). Thus, all truncated spheres in this work are spheres in some $\ell_p$ metric, and the truncated metrics are a convenient way to consider different $\ell_p$ metrics for the graph components.
This observation allows connections with other previous works such as (for just one $t$) perfect codes in the maximum metric and in the $\ell_p$ metrics
\cite{Campello}.
\end{remark}

\begin{remark}\label{remark} We start to rephrase the graph-coloring facts of \cite{rainbow}, to be completed in Subsection~\ref{s2}.
Announced as Theorem~\ref{t2} below, there is
a construction of lattice $n$-PTMC's $S$ whose induced subgraphs $[S]$ in $\Lambda_n$ have their components $H$ with:

\begin{enumerate}\item[ \bf(i)] vertex sets $V(H)$ whose convex hulls are $n$-boxes in $\mathbb{R}^n$ with
distance 3 to $[S]\setminus H$, and representatives (one per component $H$) forming a lattice with generators along the coordinate directions;
\item[\bf(ii)] each $V(H)$ contained in a truncated sphere $(H)^{\kappa(\langle H\rangle)}$ whose radius is $n$.
\end{enumerate}
(We do not know whether similar lattice $t$-PTMC's $S$ exist with $[S]$ having $r$-pa\-ral\-le\-lo\-to\-pes as their components, for $r$ fixed such that $0<r<n$.)

Extending the meaning of the adjective {\it lattice} in Subsection~\ref{FR} below, as in \cite{rainbow} we have in Theorems~\ref{t3}-\ref{t4} that a lattice $\kappa$-PTMC $S$ exists in $\Lambda_n$ whose induced subgraph $[S]$ has its components $H$ with
vertex sets $V(H)$:
\begin{enumerate}\item[\bf(i$'$)] having both $(n-1)$-cubes and $0$-cubes as their convex hulls in $\mathbb{R}^n$;
\item[\bf(ii$'$)] contained in truncated spheres $(H)^{\kappa(\langle H\rangle)}$ whose radii, namely 1 and $(n-2)$, correspond respectively to the $(n-1)$-cubes and 0-cubes
in item (i$'$).
\end{enumerate}
A set $S\subset \mathbb{Z}^n=E(\Lambda_n)$ is {\it periodic} if and only if there exist $p_1, \ldots,p_n$ in $\mathbb{Z}$ such that $v\in S$
implies $v\pm p_ie_i\in S, \forall i\in[n]$. Since each lattice $t$-PTMC$[H]$ $S$ is periodic \cite{ADH},
then every canonical projection from $\Lambda_n$ onto  a toroidal grid $\mathcal T$, i.e. a cartesian product ${\mathcal T}=C_{k_1p_1}\square C_{k_2p_2}\square\cdots\square C_{k_np_n}$ of $n$ cycles $C_{k_ip_i}$ ($0<k_i\in\mathbb{Z}, \forall i\in[n]$), takes  $S$ onto a $t$-PTMC$[H]$ in $\mathcal T$. This observation adapts to respective situations that complement the statements of Theorems~\ref{t2}-\ref{t4} via canonical projections from $\mathbb{Z}^n$ onto adequate toroidal grids $\mathcal T$.
\end{remark}

\begin{question} Do lattice $t$-PTMC's $S$ exist with $[S]$ having $r$-pa\-ral\-le\-lo\-to\-pes as their components, for $r$ fixed such that $0<r<n$?\end{question}

\subsection[\bf 1.2.]{\bf Further specifications}\label{FR}

Given a lattice $L$ in $\mathbb{Z}^n$, a subset $T\subseteq\mathbb{Z}^n$ that contains exactly one vertex in each class mod $L$ (so that $T$ is a complete system of coset representatives of $L$ in $\Lambda_n$) is said to be an
{\it FR} (acronym suggesting ``fundamental region'') of $L$. A partition of $\mathbb{Z}^n$ into FR's of $L$ is said to be a  {\it tiling} of $\mathbb{Z}^n$.
Those FR's are said to be its {\it tiles}.

We extend the notion of lattice $\kappa$-PTMC so that the associated {\it fundamental region} (see \cite{Calder}, pg 26), or FR (see above) of every new lattice contains a finite number of (in our applications, just two) members $H$ of each $\langle H\rangle$.

In the literature, existing constructions of lattice $t$-PTMC's in $\Lambda_n$ ($t<n$) concern just $t=1$ (see \cite{ADH,BE,Etzion,BH}) but there are not many known lattice $1$-PTMC's.
For example, \cite{Lucho} shows that there is only one lattice $1$-PTMC$[Q^2_2]$ and no non-lattice $1$-PTMC$[Q^2_2]$. In addition, there is a lattice 2-PDDS$[Q^2_1]$ in  $\Lambda_3$ arising from a tiling of Minkowsky cited in \cite{ADH}.

\begin{conjecture} There is no $t${\rm-PTMC}  lattice in $\Lambda_n$, for $1<t<n$.\end{conjecture}

This conjecture has an analogous form for perfect codes in the $\ell_p$ metrics in \cite{Campello}, and together with the conjecture mentioned in Remark~\ref{motiv}(B), produces a contrast with the constructions in Theorems~\ref{t3}-\ref{t4}, below.

If $S$ is a periodic non-lattice $t$-PTMC$[H]$ in $\Lambda_n$, then there exists $0<m\in\mathbb{Z}$ and a tiling of $\Lambda_n$ with tiles that are disjoint copies of the vertex set $V(H^*)$ of a connected subgraph $H^*$ induced in $\Lambda_n$ by the union of:
\begin{enumerate}\item[{\bf(a)}] the vertex sets of $m$ disjoint copies $H^1, \ldots,H^m$ of $H$ (components of $[S]$); \item[{\bf(b)}] the sets
$(H^j)^{\kappa(\langle H\rangle)}$ of vertices $v\in\mathbb{Z}^n$ with $\rho(v,H^j)\le t$, for $j\in[m]$, where

$(H^1)^{\kappa(\langle H\rangle)}, \ldots,(H^m)^{\kappa(\langle H\rangle)}$ are pairwise disjoint copies of $(H)^{\kappa(\langle H\rangle)}$ in $\mathbb{Z}^n$.\end{enumerate} By taking such an $m$ as small as possible, we say that $S$ is a {\it $t$-PTMC$[H;m]$}.

\begin{remark}\label{ref1} In Section 5 \cite{ADH}, a non-lattice $1$-PTMC$[Q^2_1;4]$ $S$ is shown to exist with a lattice $L_S$ based on it, each of the FR's of $L_S$ containing two copies of $Q^2_2$ parallel to $Oe_1$ and two more copies of $Q^2_2$ parallel to $Oe_2$, these four copies being components of $[S]$, namely generated say by the vertices $(1,0)$, $(2,0)$, $(1,3)$, $(2,3)$, $(0,1)$, $(0,2)$, $(3,1)$ and $(3,2)$. 
\end{remark}

Thus, even for a non-lattice $t$-PTMC $S$ in a $\Lambda_n$, a lattice can be recovered and formed by selected vertices $v_T$ in the corresponding tiles $T$ associated with $S$. We say that such set $S$ is a {\it lattice $t$-PTMC$[H;m]$}, indicating the number $m$ of isomorphic components of $[S]$ in a typical tile $T$ in which to fix a distinguished vertex $v_T$.

We further specify the developments above as follows. A $t$-PTMC $S$ in $\Lambda_n$ with the components of $[S]$ obtained by translations from two non-parallel                              subgraphs $H_0,H_1$ of $\Lambda_n$ is said to be a {\it $t$-PTMC$[H_0,H_1]$}.
The case of this in Remark~\ref{ref1} is a $\kappa$-PTMC,  with two translation classes of components of $[S]$ formed by the copies of $Q^2_2$ parallel to each of the directions coordinates. Here, $\kappa$ sends those copies onto $t=1$. Each tile of this $\kappa$-PTMC contains two copies of $Q^2_2$ parallel to $Oe_1$ and two copies of $Q^2_2$ parallel to $Oe_2$, accounting for $m=4$.

More generally,
let $t_i\in[n]$, for $i=0,1$. We say that a set $S\subset V$ is a {\it $(t_0,t_1)$-PTMC$[H_0,H_1]$} in $\Lambda_n$ if for each $v\in V$ there is:
\begin{enumerate}\item[{\bf(i$''$)}] a unique index $i\in\{0,1\}$ and a unique component $H_v^i$ of $[S]$ obtained by means of a translation from $H_i$ such that the truncated distance $\rho(v,H_v^i)$ from $v$ to $H_v^i$ satisfies $\rho(v,H_v^i)\leq t_i$ and \item[{\bf(ii$''$)}] a unique vertex $w$ in $H_v^i$ such that $\rho(v,w)=\rho(v,H_v^i)$.\end{enumerate}
Even though such a set $S$ is not lattice as in \cite{BH} (or lattice-like \cite{ADH}), it may happen that there exists a lattice $L_S\subset\mathbb{Z}^n$ such that for $0<m_0,m_1\in\mathbb{Z}$ there exists an FR of $L_S$ in $\Lambda_n$ given by the union of two disjoint subgraphs $H^*_0,H^*_1$, where $H^*_i$ ($i=0,1$) is induced in $\Lambda_n$ by the disjoint union of:
\begin{enumerate}\item[{\bf(a$'$)}] the vertex sets of $m_i$ disjoint copies $H_i^1, \ldots,H_i^{m_i}$ of $H_i$, (components of $[S]$);
 \item[{\bf(b$'$)}] the sets
$(H_i^j)^{\kappa(\langle H\rangle)}$ of vertices $v\in\mathbb{Z}^n$ for which $0<\rho(v,H_i^j)\le t_i$\,, for $j\in[m_i]$, where $(H_i^1)^{\kappa(\langle H\rangle)}, \ldots,(H_i^{m_i})^{\kappa(\langle H\rangle)}$ are pairwise disjoint copies of $H^\kappa(\langle H\rangle))$ in $\mathbb{Z}^n$.\end{enumerate}

\subsection[\bf 2.]{\bf Restatement of graph-coloring results}\label{s2}

A particular case of lattice $t$-PTMC$[H]$ is that in which $H$ is an $n$-box of $\Lambda_n$. For each such $H$, Theorem~\ref{t2} below says that there is a  lattice $n$-PTMC$[H]$ in $\Lambda_n$. (In \cite{BE}, $n$-boxes of unit volume in $\Lambda_n$ are shown to determine $1$-PDDS$[H]$'s if and only if either $n=2^r-1$ or $n=3^r-1$).

From Subsection~\ref{pre} we know that $S$ determines a partition of $\mathbb{Z}^n$ into spheres $W_{n,H',n}$, where $H'$ runs over the components of $[S[$. 
In each $W_{n,H',n}$, let $b_1b_2\cdots b_n$ be the vertex $a_1a_2\cdots a_n$
for which $a_1+a_2+\cdots+a_n$ is minimal. We say that this $b_1b_2\cdots b_n$ is the {\it anchor} of $W_{n,H',n}$. The anchors of the spheres $W_{n,H',n}$ form the lattice $L=L_S$. Without loss of generality we can assume that $O$ is the anchor of a $W_{n,H_0,n}$ whose truncated center $H_0$ is a component of $[S]$. Let $c_1c_2\cdots c_n$ be the vertex $a_1a_2\cdots a_n$ in $W_{n,H_0,n}$ for which $a_1+a_2+\cdots +a_n$ is maximal. Then $L_S$ has generating set $\{(1+c_1)e_1,(1+c_2)e_2,\ldots,(1+c_n)e_n\}$ and is formed by all linear combinations of those $(1+c_i)e_i$, ($i\in[n]$). 

\begin{theorem}\label{t2} \cite{rainbow}
For each $i\in[n]$, let $P_i$ be a path of length $c_i-2$, parallel to $Oe_i$, and let $H=\Pi_{i=1}^nP_i$ be an $n$-box in $\Lambda_n$. Then,
 there is a lattice $n${\rm-PTMC}$[H]$ $S$ of $\Lambda_n$ with minimum $\ell_1$-distance 3 between the components of $[S]$.
\end{theorem}

\begin{theorem}\label{t3} \cite{rainbow}
There exists a lattice $1$-{\rm PTMC}$[Q^2_2,Q^2_0;2,2]$ (in particular a {\rm PDS}) $S$ in $\Lambda_3$ with minimum $\ell_1$-distance 3 between the components of $[S]$.
\end{theorem}

To state Theorem~\ref{t4}, let $H_0=Q^2_{n-1}$, $H_1=Q^2_0$, $m_0=m_1=2$, $t_0=1$, $t_1=n-2$.

\begin{theorem}\label{t4} \label{rainbow} There exists a lattice $(1,n-2)$-{\rm PTMC}$[Q^2_{n-1},Q^2_0;2,2]$ $S$ in $\Lambda_n$.
\end{theorem}

\begin{remark}\label{apptor}
From the last observation in Remark~\ref{remark}, it can be deduced that the respective PTMC $S$ covers (via canonical projections) in Theorems~\ref{t2},  \ref{t3} and  \ref{t4}, respectively:
\begin{enumerate}
\item[$\bullet$]
an $n${\rm-PTMC}$[H]$ in
${\mathcal T}=$
$C_{c_1k_1}\square C_{c_2k_2}\square\cdots\square C_{c_nk_n}, (1<k_i\in\mathbb{Z},\forall i\in[n]);$
\item[$\bullet$]
a $1$-{\rm PTMC}$[Q^2_2,Q^2_0;2,2]$ in ${\mathcal T}=$
$C_{6k_1}\square C_{6k_2}\square C_{3k_3}, (0<k_i, \forall i\in[3]);$
\item[$\bullet$] a $(1,n-2)$-{\rm PTMC}$[Q^2_{n-1},Q^2_0;2,2]$ in  ${\mathcal T}=$
$C_{6k_1}\square\dots\square C_{6k_{n-1}}\square C_{3k_n}, (0<k_i, \forall i\in[n]).$
\end{enumerate}
\end{remark}

\section[\bf 4.]{\bf Ternary multilattice graphs}\label{s4}

Let the 2-tercube, or {\it tersquare}, $Q_2^3$ be denoted $[\emptyset]$, with vertices given by the 2-tuples $xy$, ($x,y\in F_3=\{0,1,2\}$).
As a graph, $[\emptyset]=(0,1,2)\square(0,1,2),$ namely the cartesian product of two triangles whose vertex sets are both $F_3$, is represented at the center of Figure 2 sharing a triangle with each one of 15 tersquares, to be specified below, in respective partially viewable colored backgrounds.
The {\it 1-sub-tercubes} or {\it triangles} $Q_1^3$ of the 2-tercube $[\emptyset]$ have vertex sets
$\{0y\}_{y\in F_3}$, $\{1y\}_{y\in F_3}$, $\{2y\}_{y\in F_3}$, $\{x0\}_{x\in F_3}$, $\{x1\}_{x\in F_3}$, $\{x2\}_{x\in F_3}$.
They will be indicated $x^0$, $x^1$, $x^2$, $y^0$, $y^1$, $y^2$, respectively.
To each of these triangles $t^s$ ($t\in\{x,y\};s\in F_3$) of $[\emptyset]$ we glue a corresponding tersquare $[^t_s]$ intersecting $[\emptyset]$ exactly in $t^s$.

{\center{\bf TABLE I}
$$\begin{array}{||ccccccccccc||ccccccccccc||}\hline\hline
01\!\!\!&\!\!\!-\!\!\!&\!\!\!11\!\!\!&\!\!\!-\!\!\!&\!\!\!01\!\!\!&\!\!\!-\!\!\!&\!\!\!11\!\!\!&\!\!\!-\!\!\!&\!\!\!01\!\!\!&\!\!\!-\!\!\!&\!\!\!11&02\!\!\!&\!\!\!-\!\!\!&\!\!\!22\!\!\!&\!\!\!-\!\!\!&\!\!\!02\!\!\!&\!\!\!-\!\!\!&\!\!\!22\!\!\!&\!\!\!-\!\!\!&\!\!\!02\!\!\!&\!\!\!-\!\!\!&\!\!\!22\\
 |\!\!\!&\!\!\![^{xxyy}_{0110}]\!\!\!&\!\!\!|\!\!\!&\!\!\![^{xyy}_{010}]\!\!\!&\!\!\!|\!\!\!&\!\!\![^{yy}_{10}]\!\!\!&\!\!\!|\!\!\!&\!\!\![^{xyy}_{110}]\!\!\!&\!\!\!|\!\!\!&\!\!\![^{xxyy}_{1010}]\!\!\!&\!\!\!|& |\!\!\!&\!\!\![^{xxyy}_{0220}]\!\!\!&\!\!\!|\!\!\!&\!\!\![^{xyy}_{020}]\!\!\!&\!\!\!|\!\!\!&\!\!\![^{yy}_{20}]\!\!\!&\!\!\!|\!\!\!&\!\!\![^{xyy}_{220}]\!\!\!&\!\!\!|\!\!\!&\!\!\![^{xxyy}_{2020}]\!\!\!&\!\!\!|\\
00\!\!\!&\!\!\!-\!\!\!&\!\!\!10\!\!\!&\!\!\!-\!\!\!&\!\!\!00\!\!\!&\!\!\!-\!\!\!&\!\!\!10\!\!\!&\!\!\!-\!\!\!&\!\!\!00\!\!\!&\!\!\!-\!\!\!&\!\!\!10&  00\!\!\!&\!\!\!-\!\!\!&\!\!\!20\!\!\!&\!\!\!-\!\!\!&\!\!\!00\!\!\!&\!\!\!-\!\!\!&\!\!\!20\!\!\!&\!\!\!-\!\!\!&\!\!\!00\!\!\!&\!\!\!-\!\!\!&\!\!\!20\\
 |\!\!\!&\!\!\![^{xxy}_{011}]\!\!\!&\!\!\!|\!\!\!&\!\!\![^{xy}_{01}]\!\!\!&\!\!\!|\!\!\!&\!\!\![^y_1]\!\!\!&\!\!\!|\!\!\!&\!\!\![^{xy}_{11}]\!\!\!&\!\!\!|&[^{xxy}_{101}]\!\!\!&\!\!\!|\!\!\!&\!\!\! |\!\!\!&\!\!\![^{xxy}_{022}]\!\!\!&\!\!\!|\!\!\!&\!\!\![^{xy}_{02}]\!\!\!&\!\!\!|\!\!\!&\!\!\![^{y}_{2}]\!\!\!&\!\!\!|\!\!\!&\!\!\![^{xy}_{22}]\!\!\!&\!\!\!|\!\!\!&\!\!\![^{xxy}_{202}]\!\!\!&\!\!\!|\\
01\!\!\!&\!\!\!-\!\!\!&\!\!\!11\!\!\!&\!\!\!-\!\!\!&\!\!\!01\!\!\!&\!\!\!-\!\!\!&\!\!\!11\!\!\!&\!\!\!-\!\!\!&\!\!\!01\!\!\!&\!\!\!-\!\!\!&\!\!\!11& 02\!\!\!&\!\!\!-\!\!\!&\!\!\!22\!\!\!&\!\!\!-\!\!\!&\!\!\!02\!\!\!&\!\!\!-\!\!\!&\!\!\!22\!\!\!&\!\!\!-\!\!\!&\!\!\!02\!\!\!&\!\!\!-\!\!\!&\!\!\!22\\
 |\!\!\!&\!\!\![^{xx}_{01}]\!\!\!&\!\!\!|\!\!\!&\!\!\![^x_0]\!\!\!&\!\!\!|\!\!\!&\!\!\![\emptyset]\!\!\!&\!\!\!|\!\!\!&\!\!\![^x_1]\!\!\!&\!\!\!|\!\!\!&\!\!\![^{xx}_{10}]\!\!\!&\!\!\!|& |\!\!\!&\!\!\![^{xx}_{02 }]\!\!\!&\!\!\!|\!\!\!&\!\!\![^{x}_{0}]\!\!\!&\!\!\!|\!\!\!&\!\!\![\emptyset]\!\!\!&\!\!\!|\!\!\!&\!\!\![^{x}_{2}]\!\!\!&\!\!\!|\!\!\!&\!\!\![^{xx}_{20}]\!\!\!&\!\!\!|\\
00\!\!\!&\!\!\!-\!\!\!&\!\!\!10\!\!\!&\!\!\!-\!\!\!&\!\!\!00\!\!\!&\!\!\!-\!\!\!&\!\!\!10\!\!\!&\!\!\!-\!\!\!&\!\!\!00\!\!\!&\!\!\!-\!\!\!&\!\!\!10& 00\!\!\!&\!\!\!-\!\!\!&\!\!\!20\!\!\!&\!\!\!-\!\!\!&\!\!\!00\!\!\!&\!\!\!-\!\!\!&\!\!\!20\!\!\!&\!\!\!-\!\!\!&\!\!\!00\!\!\!&\!\!\!-\!\!\!&\!\!\!20\\
 |\!\!\!&\!\!\![^{xxy}_{010}]\!\!\!&\!\!\!|\!\!\!&\!\!\![^{xy}_{00}]\!\!\!&\!\!\!|\!\!\!&\!\!\![^y_0]\!\!\!&\!\!\!|\!\!\!&\!\!\![^{xy}_{10}]\!\!\!&\!\!\!|\!\!\!&\!\!\![^{xxy}_{100}]\!\!\!&\!\!\!|& |\!\!\!&\!\!\![^{xxy}_{020}]\!\!\!&\!\!\!|\!\!\!&\!\!\![^{xy}_{00}]\!\!\!&\!\!\!|\!\!\!&\!\!\![^{y}_{0}]\!\!\!&\!\!\!|\!\!\!&\!\!\![^{xy}_{20}]\!\!\!&\!\!\!|\!\!\!&\!\!\![^{xxy}_{200}]\!\!\!&\!\!\!|\\
10\!\!\!&\!\!\!-\!\!\!&\!\!\!11\!\!\!&\!\!\!-\!\!\!&\!\!\!10\!\!\!&\!\!\!-\!\!\!&\!\!\!11\!\!\!&\!\!\!-\!\!\!&\!\!\!10\!\!\!&\!\!\!-\!\!\!&\!\!\!11&  20\!\!\!&\!\!\!-\!\!\!&\!\!\!22\!\!\!&\!\!\!-\!\!\!&\!\!\!20\!\!\!&\!\!\!-\!\!\!&\!\!\!22\!\!\!&\!\!\!-\!\!\!&\!\!\!20\!\!\!&\!\!\!-\!\!\!&\!\!\!22\\
 |\!\!\!&\!\!\![^{xxyy}_{0101}]\!\!\!&\!\!\!|\!\!\!&\!\!\![^{xyy}_{001}]\!\!\!&\!\!\!|\!\!\!&\!\!\![^{yy}_{01}]\!\!\!&\!\!\!|\!\!\!&\!\!\![^{xyy}_{101}]\!\!\!&\!\!\!|\!\!\!&\!\!\![^{xxyy}_{1001}]\!\!\!&|\!\!\!& |\!\!\!&\!\!\![^{xxyy}_{0202}]\!\!\!&\!\!\!|\!\!\!&\!\!\![^{xyy}_{002}]\!\!\!&\!\!\!|\!\!\!&\!\!\![^{yy}_{02}]\!\!\!&\!\!\!|\!\!\!&\!\!\![^{xyy}_{202}]\!\!\!&\!\!\!|\!\!\!&\!\!\![^{xxyy}_{2002}]\!\!\!&\!\!\!|\\
00\!\!\!&\!\!\!-\!\!\!&\!\!\!01\!\!\!&\!\!\!-\!\!\!&\!\!\!00\!\!\!&\!\!\!-\!\!\!&\!\!\!01\!\!\!&\!\!\!-\!\!\!&\!\!\!00\!\!\!&\!\!\!-\!\!\!&\!\!\!01&  00\!\!\!&\!\!\!-\!\!\!&\!\!\!02\!\!\!&\!\!\!-\!\!\!&\!\!\!00\!\!\!&\!\!\!-\!\!\!&\!\!\!02\!\!\!&\!\!\!-\!\!\!&\!\!\!00\!\!\!&\!\!\!-\!\!\!&\!\!\!02\\\hline\hline

12\!\!\!&\!\!\!-\!\!\!&\!\!\!22\!\!\!&\!\!\!-\!\!\!&\!\!\!12\!\!\!&\!\!\!-\!\!\!&\!\!\!22\!\!\!&\!\!\!-\!\!\!&\!\!\!12\!\!\!&\!\!\!-\!\!\!&\!\!\!22&  \!\!\!&\!\!\!\!\!\!&\!\!\!\!\!\!&\!\!\!\!\!\!&\!\!\!\!\!\!&\!\!\!\!\!\!&\!\!\!\!\!\!&\!\!\!\!\!\!&\!\!\! \!\!\!&\!\!\!\!\!\!&\!\!\!\\
 |\!\!\!&\!\!\![^{xxyy}_{1221}]\!\!\!&\!\!\!|\!\!\!&\!\!\![^{xyy}_{121}]\!\!\!&\!\!\!|\!\!\!&\!\!\![^{yy}_{21}]\!\!\!&\!\!\!|\!\!\!&\!\!\![^{xyy}_{221}]\!\!\!&\!\!\!|\!\!\!&\!\!\![^{xxyy}_{2121}]\!\!\!&\!\!\!| &\!\!\!&\!\!\!\!\!\!&\!\!\!\!\!\!&\!\!\!\!\!\!&\!\!\!\!\!\!&\!\!\!\!\!\!&\!\!\!\!\!\!&\!\!\!\!\!\!&\!\!\! \!\!\!&\!\!\!\!\!\!&\!\!\!\\
11\!\!\!&\!\!\!-\!\!\!&\!\!\!21\!\!\!&\!\!\!-\!\!\!&\!\!\!11\!\!\!&\!\!\!-\!\!\!&\!\!\!21\!\!\!&\!\!\!-\!\!\!&\!\!\!11\!\!\!&\!\!\!-\!\!\!&\!\!\!21&  &\!\!\!\!\!\!&\!\!\!22\!\!\!&\!\!\!-\!\!\!&\!\!\!02\!\!\!&\!\!\!-\!\!\!&\!\!\!12\!\!\!&\!\!\!-\!\!\!&\!\!\!22\!\!\!&\!\!\!\!\!\!&\!\!\!\\
 |\!\!\!&\!\!\![^{xxy}_{122}]\!\!\!&\!\!\!|\!\!\!&\!\!\![^{xy}_{12}]\!\!\!&\!\!\!|\!\!\!&\!\!\![^y_2]\!\!\!&\!\!\!|\!\!\!&\!\!\![^{xy}_{22}]\!\!\!&\!\!\!|\!\!\!&\!\!\![^{xxy}_{212}]\!\!\!&\!\!\!|\!\!\!&\!\!\! \!\!\!&\!\!\!\!\!\!&\!\!\!|\!\!\!&\!\!\!\!\!\!&\!\!\!|\!\!\!&\!\!\!\!\!\!&\!\!\!|\!\!\!&\!\!\!\!\!\!&\!\!\!|\!\!\!&\!\!\!\!\!\!&\!\!\!\\
12\!\!\!&\!\!\!-\!\!\!&\!\!\!22\!\!\!&\!\!\!-\!\!\!&\!\!\!12\!\!\!&\!\!\!-\!\!\!&\!\!\!22\!\!\!&\!\!\!-\!\!\!&\!\!\!12\!\!\!&\!\!\!-\!\!\!&\!\!\!22&&\!\!\!\!\!\!&\!\!\!21\!\!\!&\!\!\!-\!\!\!&\!\!\!01\!\!\!&\!\!\!-\!\!\!&\!\!\!11\!\!\!&\!\!\!-\!\!\!&\!\!\!21\!\!\!&\!\!\!\!\!\!&\!\!\!\\
 |\!\!\!&\!\!\![^{xx}_{12}]\!\!\!&\!\!\!|\!\!\!&\!\!\![^x_1]\!\!\!&\!\!\!|\!\!\!&\!\!\![\emptyset]\!\!\!&\!\!\!|\!\!\!&\!\!\![^x_2]\!\!\!&\!\!\!|\!\!\!&\!\!\![^{xx}_{21}]\!\!\!&\!\!\!|\!\!\!&\!\!\! \!\!\!&\!\!\!\!\!\!&\!\!\!|\!\!\!&\!\!\!\!\!\!&\!\!\!|\!\!\!&\!\!\![\emptyset]\!\!\!&\!\!\!|\!\!\!&\!\!\!\!\!\!&\!\!\!|\!\!\!&\!\!\!\!\!\!&\!\!\!\\
11\!\!\!&\!\!\!-\!\!\!&\!\!\!21\!\!\!&\!\!\!-\!\!\!&\!\!\!11\!\!\!&\!\!\!-\!\!\!&\!\!\!21\!\!\!&\!\!\!-\!\!\!&\!\!\!11\!\!\!&\!\!\!-\!\!\!&\!\!\!21
&&\!\!\!\!\!\!&\!\!\!20\!\!\!&\!\!\!-\!\!\!&\!\!\!00\!\!\!&\!\!\!-\!\!\!&\!\!\!10\!\!\!&\!\!\!-\!\!\!&\!\!\!20\!\!\!&\!\!\!\!\!\!&\!\!\!\\
 |\!\!\!&\!\!\![^{xxy}_{121}]\!\!\!&\!\!\!|\!\!\!&\!\!\![^{xy}_{11}]\!\!\!&\!\!\!|\!\!\!&\!\!\![^y_1]\!\!\!&\!\!\!|\!\!\!&\!\!\![^{xy}_{21}]\!\!\!&\!\!\!|\!\!\!&\!\!\![^{xxy}_{211}]\!\!\!&\!\!\!|\!\!\!&\!\!\! \!\!\!&\!\!\!\!\!\!&\!\!\!|\!\!\!&\!\!\!\!\!\!&\!\!\!|\!\!\!&\!\!\!\!\!\!&\!\!\!|\!\!\!&\!\!\!\!\!\!&\!\!\!|\!\!\!&\!\!\!\!\!\!&\!\!\!\\
12\!\!\!&\!\!\!-\!\!\!&\!\!\!22\!\!\!&\!\!\!-\!\!\!&\!\!\!12\!\!\!&\!\!\!-\!\!\!&\!\!\!22\!\!\!&\!\!\!-\!\!\!&\!\!\!12\!\!\!&\!\!\!-\!\!\!&\!\!\!22
&&\!\!\!\!\!\!&\!\!\!22\!\!\!&\!\!\!-\!\!\!&\!\!\!02\!\!\!&\!\!\!-\!\!\!&\!\!\!12\!\!\!&\!\!\!-\!\!\!&\!\!\!22\!\!\!&\!\!\!\!\!\!&\!\!\!\\
 |\!\!\!&\!\!\![^{xxyy}_{1212}]\!\!\!&\!\!\!|\!\!\!&\!\!\![^{xyy}_{112}]\!\!\!&\!\!\!|\!\!\!&\!\!\![^{xy}_{12}]\!\!\!&\!\!\!|\!\!\!&\!\!\![^{xyy}_{212}]\!\!\!&\!\!\!|\!\!\!&\!\!\![^{xxyy}_{2112}]\!\!\!&\!\!\!|\!\!\!&\!\!\! \!\!\!&\!\!\!\!\!\!&\!\!\!\!\!\!&\!\!\!\!\!\!&\!\!\!\!\!\!&\!\!\!\!\!\!&\!\!\!\!\!\!&\!\!\!\!\!\!&\!\!\! \!\!\!&\!\!\!\!\!\!&\!\!\!\\
11\!\!\!&\!\!\!-\!\!\!&\!\!\!21\!\!\!&\!\!\!-\!\!\!&\!\!\!11\!\!\!&\!\!\!-\!\!\!&\!\!\!21\!\!\!&\!\!\!-\!\!\!&\!\!\!11\!\!\!&\!\!\!-\!\!\!&\!\!\!21 &&\!\!\!\!\!\!&\!\!\!\!\!\!&\!\!\!\!\!\!&\!\!\!\!\!\!&\!\!\!\!\!\!&\!\!\!\!\!\!&\!\!\!\!\!\!&\!\!\! \!\!\!&\!\!\!\!\!\!&\!\!\!\\\hline\hline
\end{array}$$}

This way, six tersquares $[^x_0], \ldots,[^y_2]$ are obtained (to be called {\it subcentral tersquares}, colored yellow and light-blue in Figure 2) that intersect $[\emptyset]$ respectively in the triangles $x^0, \ldots,y^2$.
Next, we glue to the remaining new triangles (i.e., other than those already in $[\emptyset]$), a set of nine new tersquares that we denote $[^{xy}_{00}], \ldots,[^{xy}_{22}]$ (and call {\it corner tersquares}, colored green and gray in Figure 2). These nine tersquares share merely a vertex with $[\emptyset]$ and their disposition with respect to the central tersquare $[\emptyset]$ and subcentral tersquares $x^0, \ldots,y^2$ is specified in Figure 2. 
In fact, $[^{xy}_{ij}]$ shares a triangle $y^j$ with $[^x_i]$ and a triangle $x^i$ with $[^y_j]$, for $i,j=0,1,2$. Such triangles can be further specified respectively as $y^j[^x_i]$ and $x^i[^y_j]$.

The graph $[[\emptyset]]$ resulting from the union of the $1+6+9=16$ tersquares constructed so far, will be referred to as the 2-{\it hive} $[[\emptyset]]$. The {\it central tersquare} $[\emptyset]$ of $[[\emptyset]]$ shares just one triangle with each of the six subcentral tercubes and just one vertex with each of the nine corner tersquares.
The tersquare $[\emptyset]$, given as a subgraph at the center of Figure 2, also results in the representation at the lower-right quarter of Table I by identifying
the quadruple of vertices labelled 22, as well as each pair of vertices labelled 20,  21,  02 and 12. 

The construction above continues with the iterative glueing of tersquares along {\it free} triangles, namely those triangles along which glueing of a new continuing tersquare was not already performed. In the limit of feasible glueings, a total compound graph$\Gamma_2$ of tersquares, glued in pairs of adjacent  tersquares along corresponding triangles, results in an infinite locally finite graph, the {\it ternary square compound} graph $\Gamma_2$, that we also call {\it ternary non-lattice multilattice graph}. Notation for all tersquares and triangles of $\Gamma_2$, continuing the notation of the 16 tersquares and their triangles above, is completed below. The word ``multilattice" here corresponds to the fact that, even though $\Gamma_n$ is not a lattice graph itself, it is such that each of its binary $n$-cubes is contained in a unique maximal lattice of $\Gamma_n$.

\begin{figure}[htp]
\hspace*{3mm}
\includegraphics[scale=0.35]{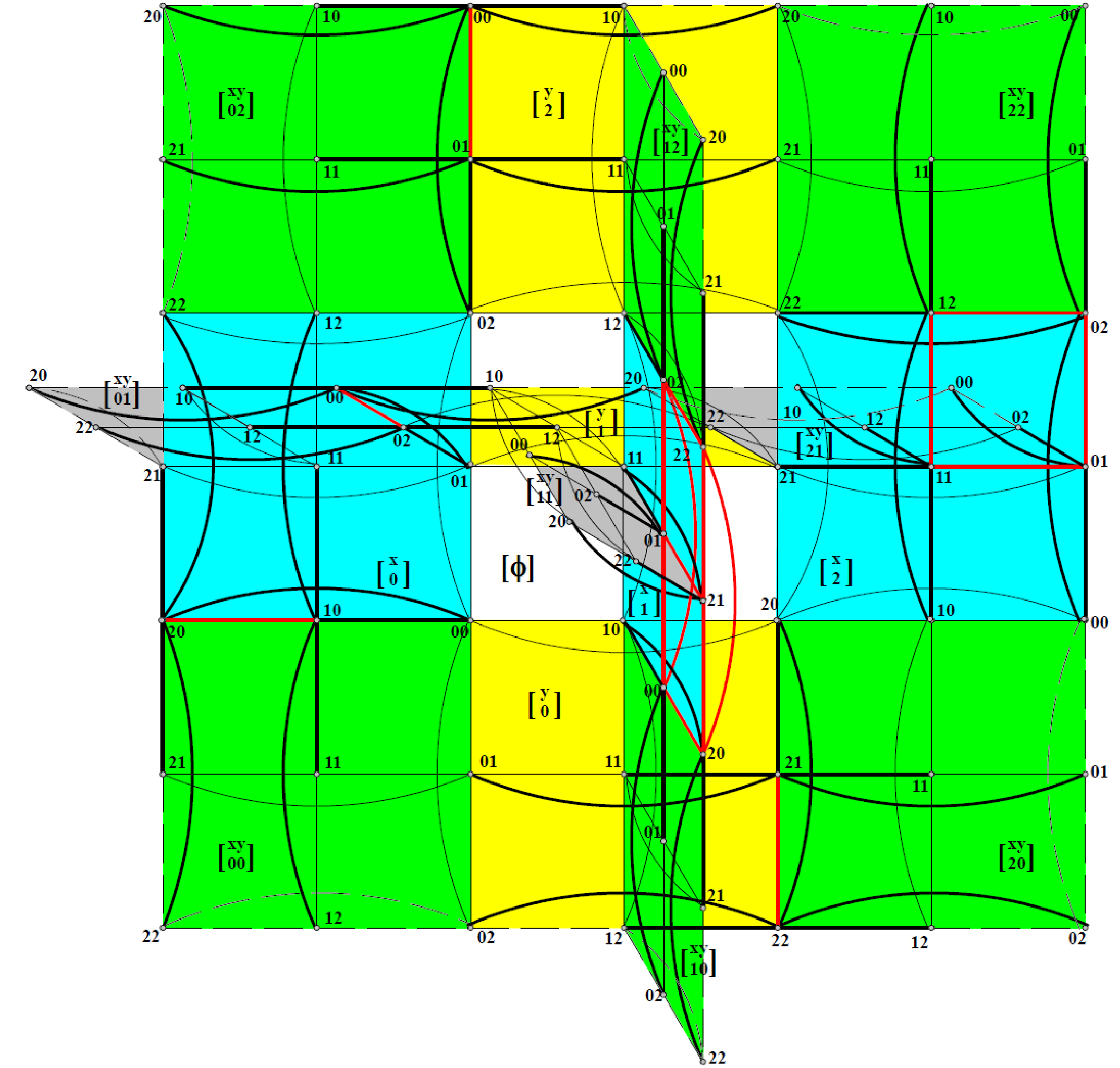}
\caption{Representing a PDS of $[[\emptyset]]$ in $\Gamma_2$}
\end{figure}

Three subgraphs of $\Gamma_2$ are represented in Table I, containing just one 4-cycle (of the 9 in $Q_2^3$) per participating (but not totally shown) glued tersquare. This yields 9 sub-lattice types in $\Gamma_2$ (one per 4-cycle of $Q_2^3$, three partially shown in Table I), each isomorphic to $\Lambda_2$. Thus, $\Gamma_2$ is a superset of each of the copies of $\Lambda_2$ generated by some 4-cycle of $\Gamma_2$. We further specify $\Gamma_2$ after presenting a non-isolated PDS in $[[\emptyset]]$ (Theorem~\ref{6c}).

\begin{figure}[htp]
\hspace*{21mm}
\includegraphics[scale=0.54]{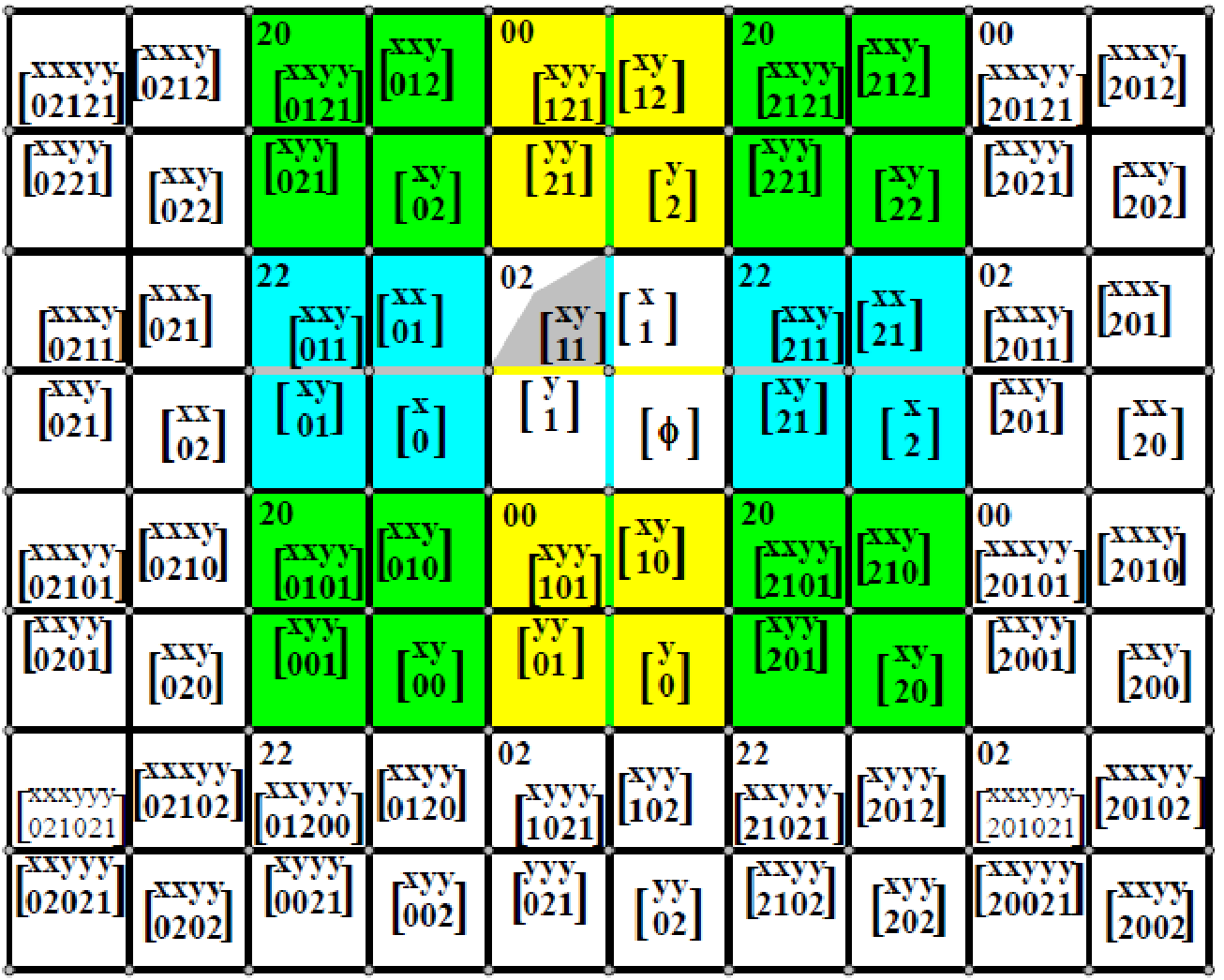}
\caption{Planar projection of a portion of $\Gamma_2$ larger than in Figure 2}
\end{figure}

The large connected graph in Figure 2 represents the 2-hive $[[\emptyset]]$, with $[^x_0]$, $[^x_1]$ and $[^x_2]$ in light-blue background, $[^y_0]$, $[^y_1]$ and $[^x_2]$ in yellow background, and the nine tersquares $[^{xy}_{i\,j}]$ in green background if $y\ne 1$ and light-gray background if $y=1$, where the apparent colored areas spatial disposition results in the partial hiding of some of these areas.
The thick red edges induce a subgraph of $[[\emptyset]]$ whose vertex set is a non-isolated PDS $S$ (as defined in Remark~\ref{ter}) of $[[\emptyset]]$. The dominating edges are drawn in thick black trace, allowing the reader to verify that all vertices of $[[\emptyset]]$ not in $S$ are dominated as indicated in Remark~\ref{ter}.
(It is elementary to verify that no isolated PDS in $[[\emptyset]]$ exists). The remaining edges of $[[\emptyset]]$, not in thick trace, are in dashed  trace if they are ``exterior'' edges of $[[\emptyset]]$ (i.e., those shared by 2-hives other than $[[\emptyset]]$ in $\Gamma_2$) and in thin trace otherwise (i.e., if they are ``interior'' edges of $[[\emptyset]]$).
Motivation for pursuing such a non-isolated PDS arose as a palliative remedy for the impossibility of having here a PDS like the Livingston-Stout PDS in the grid $P_4\square P_4$ \cite{LS}, which happens to be the only isolated PDS in grid graphs $P_m\square P_n$, for $2<min\{m,n\}$ \cite{DD}.

\begin{theorem}\label{6c}
There exists a non-isolated PDS, or 1-PTMC$[Q_1^2,Q_2^2,Q_1^3\square Q_1^2;4,$ $1,1]$,
in the 2-hive $[[\emptyset]]\subset\Gamma_2$.
\end{theorem}

\begin{proof}
The claimed PDS in $[[\emptyset]]$ is formed by the vertex sets of the following components (of its induced subgraphs):
\begin{enumerate}
\item[\bf(a)] the edge $(00,01)$ of triangle $(00,01,02)$ shared by tersquares $[^y_2]$ and $[^{xy}_{02}]$;
\item[\bf(b)] the edge $(10,20)$ of triangle $(00,10,20)$ shared by tersquares $[^x_0]$ and $[^{xy}_{00}]$;
\item[\bf(c)] the edge $(21,22)$ of triangle $(20,21,22)$ shared by tersquares $[^y_0]$ and $[^{xy}_{20}]$;
\item[\bf(d)] the edge $(00,02)$ of triangle $(00,02,01)$ shared by tersquares $[^y_1]$ and $[^{xy}_{01}]$;
\item[\bf(e)] the 4-cycle $(02,12,11,01)$ of tersquare $[^x_2]$, where its edge $(02,12)$ (of triangle
$\Delta=(02,12,22)$) is shared (with $\Delta$) with tersquare $[^{xy}_{22}]$;

\item[\bf(f)] the triangular prism in $[^x_1]$ formed by triangles $(00,01,02)$ and $(20,$
$21,22)$ and edges $(00,20)$,
$(02,22)$ (of triangle $(02,22,12)$, shared by tersquare
$[^{xy}_{22}]$) and $(01,21)$ (of triangle $(01,21,11)$, shared by tersquare $[^{xy}_{11}]$).
\end{enumerate}

\begin{figure}[htp]
\hspace*{3.5mm}
\includegraphics[scale=0.35]{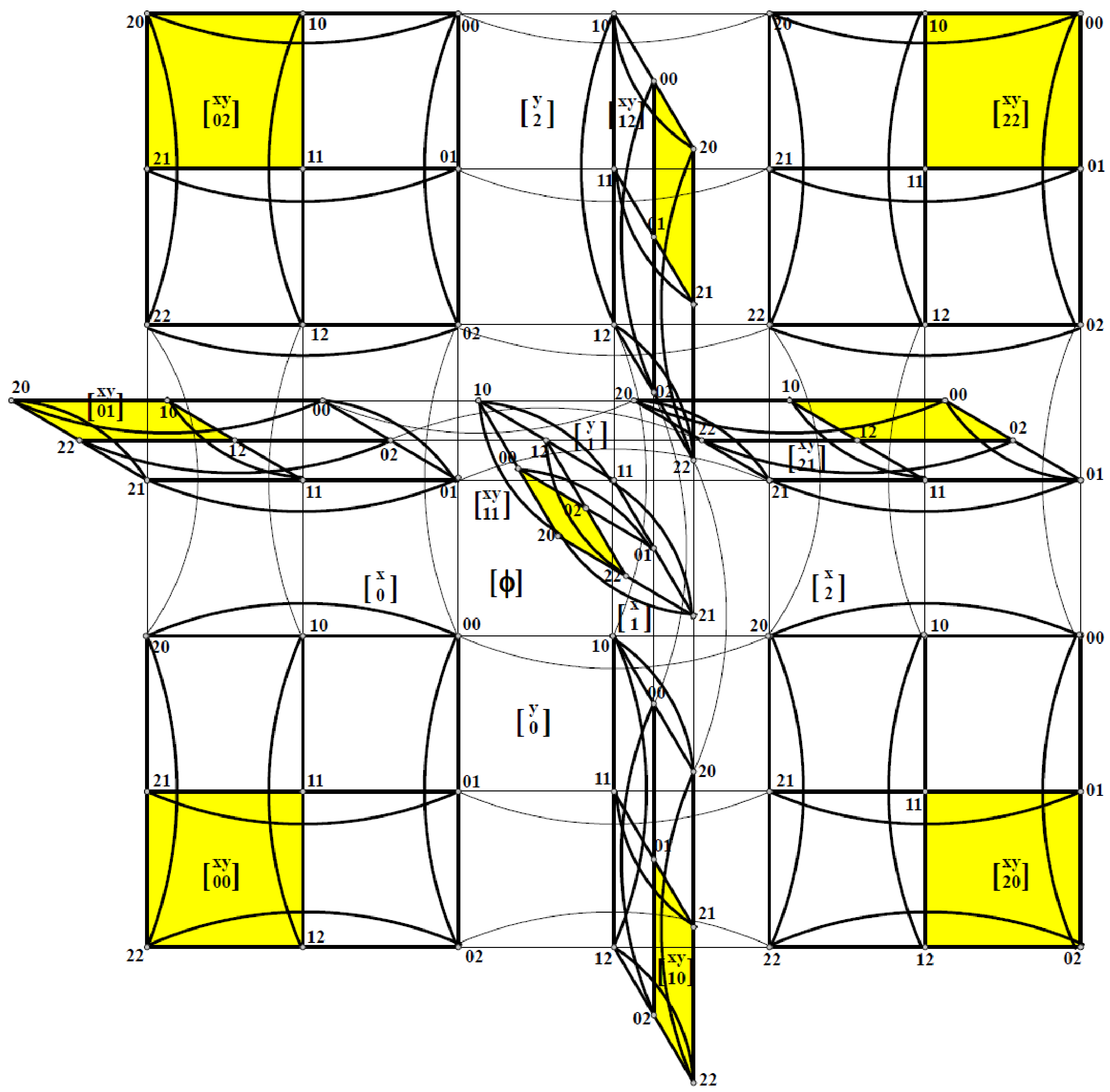}
\caption{Components of a 2-PTMC$[Q_2^2]$ of $[[\emptyset]]$ in $\Gamma_2$}
\end{figure}

As can be verified in Figure 2, the edges departing from these red-edge components cover all other vertices of $[[\emptyset]]$, proving the theorem.\end{proof}

Figure 2 may be augmented by adding the tersquares $[^{xx\,y}_{i\,j\,k}]$, with $i\ne j$ in $F_3$, etc. One may continue by adding tersquares $[^{xx\cdots xx\,yy\cdots yy}_{i\,j\,\cdots kl\,mn\cdots pq}]$ with $i\ne j\ne\ldots\ne k\ne l$ (resp. $m\ne n\ne\ldots\ne p\ne q$), i.e. not two contiguous equal values under the $x$'s (resp. $y$'s).
Taking $i,j,\ldots,k,l,m,n,\ldots p,q$ in $F_3$ to be contiguously different with just two values, (e.g. 0,1; resp. 0,2; resp. 1,2, in the upper-left; resp. upper-right; resp. lower-left quarter in Table I) is a way of obtaining 3 sub-lattices of $\Gamma_2$, each isomorphic to $\Lambda_2$, as in Table I.
But there is an infinite family of ``parallel'' sub-lattices in $\Gamma_2$ for each of the 9 sublattice types, one per each 4-cycle of a typical tersquare as in the lower-left quarter of Table I. For example, the type with values 0,2 represents ``parallel'' sub-lattices having ``shortest'' tersquares $[\emptyset]$, $[^x_1]$, $[^y_1]$, $[^{xx}_{01}]$, $[^{yy}_{01}]$, $[^{xx}_{21}]$, $[^{yy}_{21}]$, etc., including every $[^{xx\cdots xx}_{i\,j\,\cdots kl}]$, $[^{\,y\,y\cdots y\,y}_{mn\cdots pq}]$ and $[^{xx\cdots xx\,\,yy\cdots yy}_{i\,j\,\cdots kl\,mn\cdots pq}]$.

These sub-lattices can be refined by replacing each edge $e$ in them by the 2-path closing a triangle with $e$ and adding more vertices and additional edges out of them to the new vertices obtained by the said replacement, in order to distinguish the glued $2$-tercubes in $\Gamma_2$. In particular, Figure 2 is obtained from such a procedure by starting from a plane containing the upper-right quarter of Table I. Figure 3 shows a projection of a partial extension of Figure 2 in such a plane, where tersquare colors are kept as in Figure 2, including if they are projected into a 2-path. Colors here still are light-blue, blue, green and light-gray. Each tersquare in Figure 3, represented by four squares with the common central vertex $11$, is just recognizable by its upper-left vertex denomination presence. Of the four such squares, the lower-right one has the corresponding tersquare denomination. The other tersquare denominations in the three remaining squares correspond to the counterpart denominations in the representation of $[\emptyset]$.

\begin{definition} A {\it ternary $n$-cube compound} graph $\Gamma_n$ is defined as the union of all necessary $n$-tercubes that appeared glued along their (codimension 1) ternary $(n-1)$-subtercubes, for any $2<n\in\mathbb{Z}$. 
The vertices of such $\Gamma_n$ are given in terms of the $n$-tercubes. These, that can be denoted $$[^x_0],[^y_0],\ldots,[^z_0],[^x_1],[^y_1],\ldots,[^z_1],[^x_2],[^y_2],\ldots,[^z_2],[^{xy}_{00}],\ldots$$ including all those of the general form
$[J]=[^{xx\cdots xx\,yy\cdots yy\,\cdots\cdots\cdots\, zz\cdots zz}_{i\,j\,\cdots kl\,mn\cdots pq\,\cdots\cdots\cdots\,rs\cdots vw}]$, where $x,y,\ldots,z$ represent the $n$ (ternary) coordinate directions. Such $n$-tercubes are assigned locally by reflection on the $(n-1)$-subtercubes. Thus the said vertices can be denoted $x^0[J],\ldots,z^2[J]$, by following the notations above.\end{definition}

\begin{example} Figure 2 for $n=2$, represents the 2-hive $[[\emptyset]]$ with $[\emptyset]$ at its center, and its neighboring tersquares $[^x_0]$, $[^x_1]$, $[^x_2]$ in light-blue background, $[^y_0]$, and $[^y_1]$, $[^y_2]$ in yellow background, and the tersquares neighboring those tersquares, namely $[^{xy}_{i\,j}]$, ($i,j\in F_3$), in green and light-gray backgrounds. It can be seen that $\Gamma_n$ may be considered as a superset of $\Lambda_n$ in three different ways, as was commented above in relation to Table I for $n=2$.\end{example}

We pose the following.

\begin{question}
Does there exist a non-isolated PDS $S$ in $\Gamma_n$, for $n\ge 2$? If so, could $S$ behave like a lattice, for example by  restricting itself to a lattice PDS over any sub-lattice of $\Gamma_n$, as exemplified in Table I?
\end{question}

\section[\bf 5.]{\bf Isolated 2-PTMC's in $\Gamma_2$}\label{s5}

In this section, we keep working in $\Gamma_2$, but slightly modifying the definition of a $\kappa$-PTMC by replacing the used Hamming distance $h$ by the graph distance of $\Gamma_2$. It is clear that the Hamming distance in our graph-theoretical context is just the graph distance. Then, we have the following result.

\begin{theorem}\label{novel}
There exists 262144 isolated 2-PTMC's in the 2-hive $[[\emptyset]]\subset\Gamma_2$.
\end{theorem}

\begin{proof}
In Figure 4, the nine copies $[^{xy}_{ij}]$, ($i,j\in F_3$), of the tersquare are shown with its edges in thick black trace, against the thin black trace of the remaining edges of the 2-hive $[[\emptyset]]$. These nine copies happen to be the truncated 2-spheres centered at the vertices of an isolated PTMC of $[[\emptyset]]$ constituted as a selection of one vertex per yellow-faced 4-cycle in the figure. We may say that these yellow-faced 4-cycles are the ''external'' 4-cycles of $[[\emptyset]]$. So, there are $4^9=262144$ 2-PTMC's in $\Gamma_2$.
\end{proof}

\begin{figure}[htp]
\hspace*{22mm}
\includegraphics[scale=0.54]{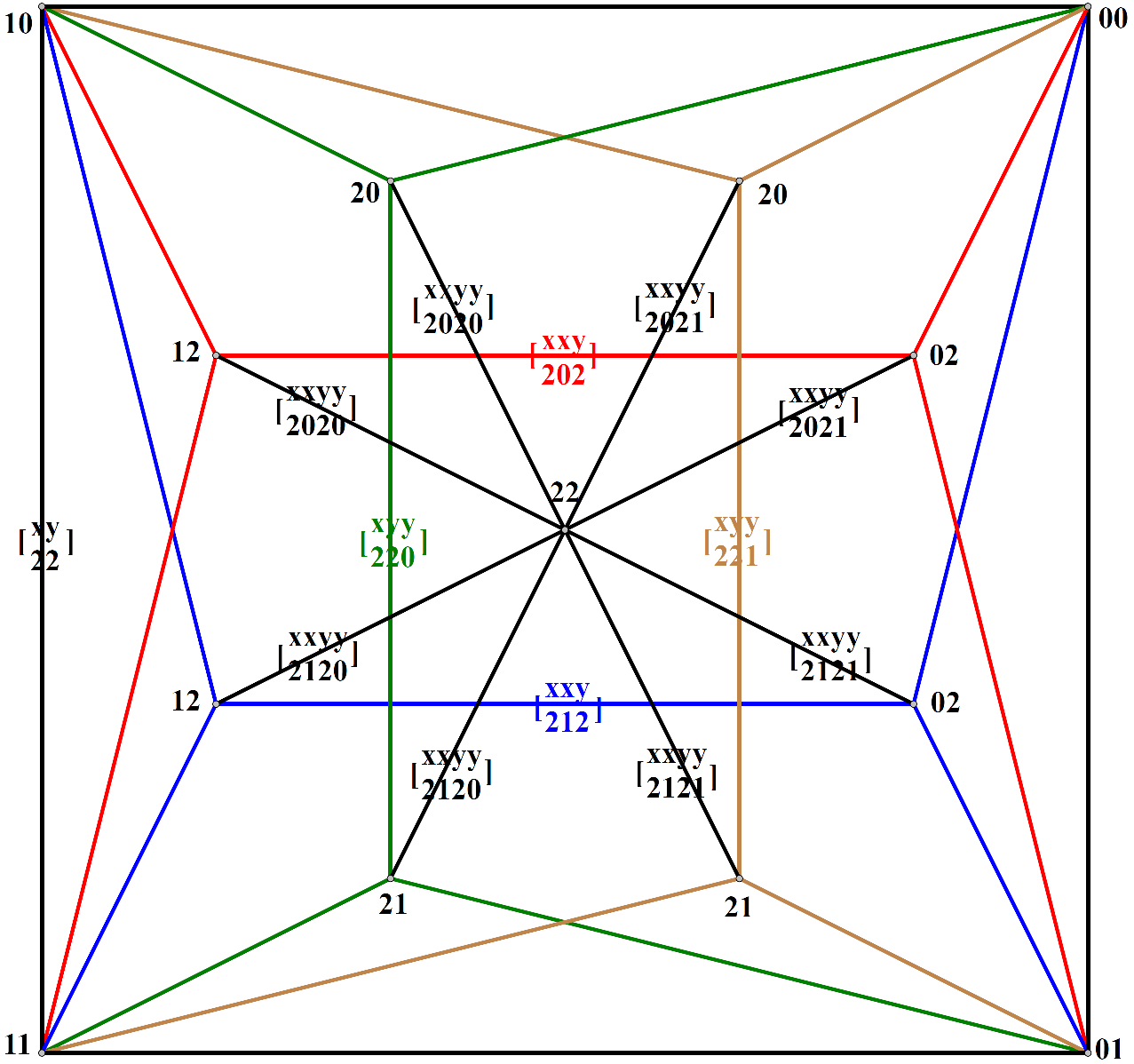}
\caption{Analysis of the square $(00,01,11,10)$ in tersquare $[^{xy}_{22}]$}.
\end{figure}

\begin{corollary}
There exists a 2-PTMC$[Q_2^2]$ in the 2-hive $[[\emptyset]]$.
\end{corollary}

\begin{proof}
The yellow-faced 4-cycles are copies of $Q_2^2$ and induce the components of the claimed 2--PTMC-$[Q_2^2]$.
\end{proof}

\begin{theorem}\label{2PTMC}
There exists an infinite number of isolated 2-PTMC in $\Gamma_2$.
\end{theorem}

\begin{proof}
Let us analyze one of the yellow squares in Figure 4, say square (00, 01,11,10) in tersquare $[^{xy}_{22}]$ , as represented in Figure 5; the remaining eight yellow squares have similar properties to those to be described, once the notations of local vertices and neighboring tercubes are updated in each case in place of those in the figure. Besides the ``external'' 4-cycle $(00,01,11,10)$  in Figure 5, whose edges are in black color, we set in red and blue, (green and hazel) colors those additional edges defining  a pair of 4-cycles in the respective tercubes $[^{xxy}_{202}]$ and $[^{xxy}_{212}]$, ($[^{xyy}_{220}]$ and $[^{xyy}_{221}]$), namely $\{(00,02,12,10),(01,02,12,11)\}$, ($\{(00,20,21,01),(10,20,21,11)\}$), equally denoted but pertaining to different tersquares, as shown. 

Since it is seen that the maximum truncated distance in Figure 5 is 2, we conclude that any one of the four vertices of the black 4-cycle dominates the remaining vertices in the other 4-cycles in the figure, in positions in their tersquares similar to those of the yellow cycles in Figure 4, namely the four just mentioned tersquares as well as the ``opposite'' tersquares $[^{xxyy}_{2020}]$,$[^{xxyy}_{2021}]$, $[^{xxyy}_{2120}]$ and $[^{xxyy}_{2121}]$, where these involved tersquare denominations are set in Figure 5 straddling edges common to two of their  4-cycles, for better reference, twice for the last four mentioned tersquares. This particular case is easily generalized to conclude that each vertex in a yellow square dominates via $\rho$ in the neighboring tersquares in a form similar to that of $[[\emptyset]]$.

With the same structure as the 2-hive $[[\emptyset]]$, we have the 2-hives of its neighboring 2-hive pair $\{[[^{xxx}_{010}]], [[^{xxx}_{012}]]\}$, etc., sharing four tersquares with $[[\emptyset]]$, tersquares denoted $(10,11,12)$ ,  both in $[^{x}_{0}]$ and in $[^{xy}_{00}]$  and $[^{xy}_{01}]$ and $[^{xy}_{02}]$. These four tersquares belong both in $[[^{xxx}_{010}]]$  and $[[^{xxx}_{012}]]$ to the four respective tercubes $[^{xx}_{01}]$, $[^{xxy}_{010}]$, $[^{xxy}_{011}]$ and $[^{xxy}_{012}]$. Moreover, the union of these four tercubes constitutes the intersection $[[^{xxx}_{010}]]\cap[[^{xxx}_{012}]]$. By considering in each of these two 2-hives the equivalent of the yellow squares of $[[\emptyset]]$ in Figure 4, and selecting in each such yellow square a vertex contributing to the construction of an isolated 2-PTMC of $\Gamma_2$,, iteration of such a procedure eventually allows the completion of such dominating set. Observe that the argument exposed from the 2-hive pair $\{[[^{xxx}_{010}]], [[^{xxx}_{012}]]\}$ is one of twelve cases arising from the 2-hive
$[[\emptyset]]$. Iteration of these twelve cases departing from any 2-hive of $\Gamma_2$ that is reached in the continuing procedure allows to advance a further step in the construction of such isolated 2-PTMC.
\end{proof}

\begin{remark} The argument of Theorem~\ref{2PTMC} can be modified to ascertain the existence of 2-PTMC$[H]$, where $H$ is a 4-cycle, or the union of  yellow squares with a common vertex in different 2-hives, like the eight squares, in Figure 5,  at vertex 10 in the tersquares $[^{xy}_{2,2}]$, $[^{xxy}_{202}]$. $[^{xxy}_{212}]$, $[^{xyy}_{220}]$, $[^{xyy}_{221}]$, $[^{xxyy}_{2020}]$, $^{xxyy}_{2021}]$ and $[^{xxyy}_{2120}]$.
\end{remark}

\begin{conjecture} There exists an isolated $n$-PTMC in $\Gamma_n$, $\forall n>2$.
\end{conjecture}

\begin{question}
Does there exist a suitable way of declaring a susbset $S\subset V(\Gamma_n)$ to be {\it periodic}, or {\it periodic-like}, that extends the notion of periodicity of $\Lambda_n$ at the end of Remark~\ref{remark} to one in $\Gamma_n$? so that a replacement of the notion of ``quotient'' or ``toroidal'' graphs of $\Lambda_n$ can be found that way for $\Gamma_n$?
\end{question}

\section{Appendix}\label{appendix}

On the left of Figure 6, an oriented representation of a ternary square is provided. Here, one may take the vertices as decision conditional diamonds, the outgoing arrows as ``Yes" or ``True" option arrows and the incoming arrows as ``No" or ``False" option arrows. An elementary flowchart modeled on this oriented ternary square is given as an example on the right of Figure 6, where the root of the flowchart corresponds to vertex 00 and there are four terminators, corresponding to the vertices 01, 02, 11 and 12. The decision conditional diamonds correspond to the vertices 00, 10, 20, 21 and 22. To avoid terminators that are the end of more than one process, as is the case of 11 and 12 in the figure, one may select in $\Gamma_2$ modeling flowchart continuation in adjacent tercubes to the immediately previous ones, each sharing with the current tercube an already used triangle.

\begin{figure}[htp]
\hspace*{2.3cm}
\includegraphics[scale=0.51]{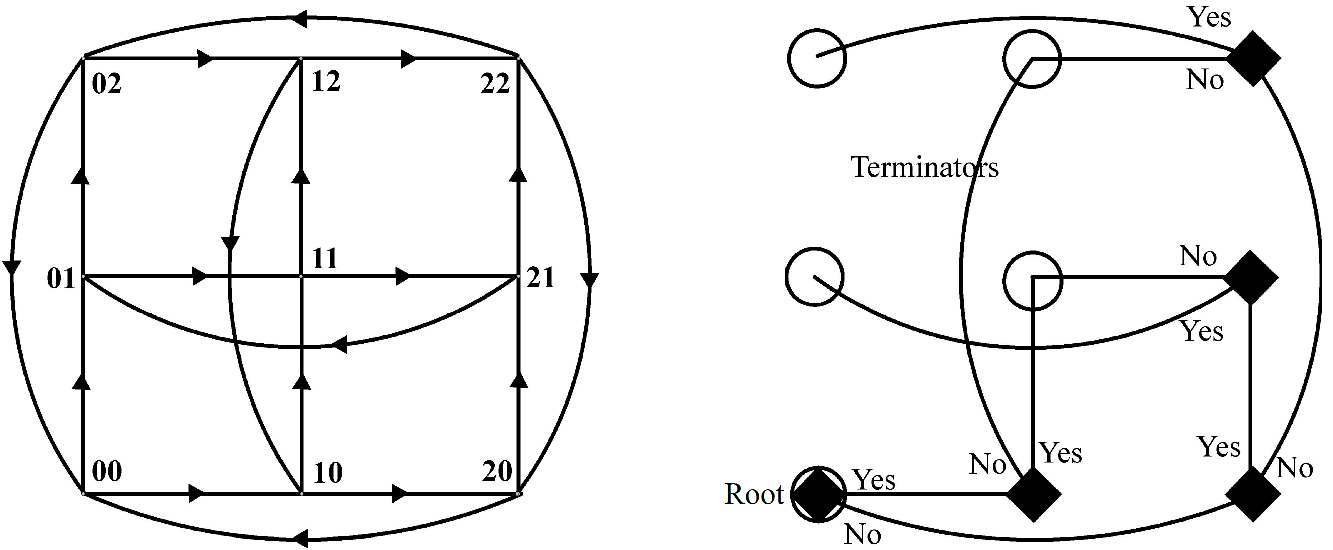}
\caption{Ternary square as a model of algorithm flowchart}
\end{figure}

\end{document}